\documentclass[sn-mathphys,Numbered]{sn-jnl}% Math and Physical Sciences Reference Style
\usepackage{graphicx}%
\usepackage{multirow}%
\usepackage{amsmath,amssymb,amsfonts}%
\usepackage{amsthm}%
\usepackage{mathrsfs}%
\usepackage[title]{appendix}%
\usepackage{xcolor}%
\usepackage{textcomp}%
\usepackage{manyfoot}%
\usepackage{booktabs}%
\usepackage{algorithm}%
\usepackage{algorithmicx}%
\usepackage{algpseudocode}%
\usepackage{listings}%
\usepackage{fancyvrb}%
\theoremstyle{thmstyleone}%
\newtheorem{proposition}{Proposition}% 
\newtheorem{definition}{Definition}
\newtheorem{corollary}{Corollary}

\begin{document}

\title{Runge--Kutta methods determined from extended phase space methods for Hamiltonian systems}
\author{\fnm{Robert I.} \sur{McLachlan}}\email{r.mclachlan@massey.ac.nz}
\affil{\orgdiv{School of Mathematical and Computational Sciences}, \orgname{Massey University}, \orgaddress{\city{Palmerston North}, \postcode{4472}, \country{New Zealand}}}

\VerbatimFootnotes

\abstract{ We study two existing extended phase space integrators for Hamiltonian systems, the {\em midpoint projection method} and the {\em symmetric projection method}, showing that the first is a pseudosymplectic and pseudosymmetric Runge--Kutta method and the second is a monoimplicit symplectic Runge--Kutta method.
}

\maketitle

\section{Introduction}
Many commonly used numerical methods for the time integration of differential equations can be expanded in B-series which elucidate their geometric and numerical properties \cite{butcher,mk2}. However, symplectic integrators with a B-series are implicit, the implicit midpoint rule being a central example \cite{hlw}. Explicit symplectic integrators exist for some systems, such as  separable classical mechanical systems \cite{mc-qu}. To avoid this restriction, Pihajoki \cite{pihojoki} introduced {\em extended phase space methods}: a new Hamiltonian, defined on the product of two copies of the original phase space, is constructed, that is amenable to explicit symplectic integration in the extended phase space.

In place of the Hamiltonian system $X_H$ associated with the Hamiltonian $H$ and canonical symplectic form $\omega$,
\begin{equation}
\label{eq:ham}
\dot q = D_2H(q,p),\quad \dot p = -D_1 H(q,p)
\end{equation}
($(q,p)\in\mathbb{R}^{2d}$),
Pihajoki considered the extended system
$$ \dot q = D_2H(x,p),\quad \dot p = -D_1H(q,y)$$
$$ \dot x = D_2H(q,y),\quad \dot y = -D_1H(x,p)$$
($(q,x,p,y)\in\mathbb{R}^{4d}$)
with initial condition
$$ (q(0),x(0),p(0),y(0))=(q_0,x_0,p_0,y_0) = (q_0,q_0,p_0,p_0)$$
such that the solution obeys $q(t)=x(t)$ and $p(t)=y(t)$ for all $t$ and $(q(t),p(t))$ satisfies the original system \eqref{eq:ham} with initial condition
$(q(0),p(0))=(q_0,p_0)$.
The extended system is Hamiltonian with extended Hamiltonian $\hat H = \hat H_A + \hat B_B$, 
$\hat H_A = H(x,p)$, $\hat H_B = H(q,y)$ and symplectic form
$\hat\omega := dq\wedge dp + dx\wedge dy$. 

As $x$ and $p$ (resp. $q$ and $y$) commute, the flow $\exp\big(t X_{\hat H_A}\big)$ of Hamilton's equations for $\hat H_A$ is given explicitly by Euler's method:
\begin{align*}
q(t) &= q_0 + t D_2 H(x_0,p_0) \\
x(t) &= x_0 \\
p(t) &= p_0 \\
y(t) &= y_0 - t D_1 H(x_0,p_0)
\end{align*}
(and analogously for $\hat H_B$).
The integrator
$$\Phi_h := \exp\big(\frac{1}{2}h X_{\hat H_A}\big) \exp\big(h X_{\hat H_B}\big) \exp\big(\frac{1}{2}h X_{\hat H_A}\big) $$
is therefore explicit, second order, and preserves the extended symplectic form $\hat\omega$.

Two key issues, however, immediately arise: the `duplicate' point $(x,y)$ may move away from the `base' point $(q,p)$; and it is not clear how symplecticity in the extended phase space is advantageous in the original phase space. Tao \cite{tao} addressed the first point by adding a coupling term $\frac{1}{2}\alpha(\|x-q\|^2 + \|y-p\|^2)$ to the extended Hamiltonian, finding that this could suppress the growth of $(q-x,p-y)$. Other authors have addressed the second point by projecting the solution to the original phase space in different ways.
Two of these, the {\em symmetric projection method} of Ohsawa \cite{ohsawa} and the {\em midpoint projection method} of Luo et al. \cite{luo}, are the subject of this paper.

The midpoint projection method, considered in Section 2, is shown to be equivalent to an explicit Runge--Kutta method that is pseudosymplectic (that is, approximately symplectic) and pseudosymmetric up to surprisingly high order -- order 5 for the leapfrog-based method of classical order 2, and order 9 for the methods of classical order 4. We suggest that these properties account for the methods' good performance in astrophysical applications.

The symmetric projection methods, considered in Section 3, are shown to be equivalent to monoimplicit symplectic Runge--Kutta methods, revealing their affine equivariance and generality.

\section{The midpoint projection method}
Luo et al. \cite{luo} composed such an extended phase space integrator with the midpoint projection\footnote{Called the midpoint permutation in \cite{luo}.}
$$ \pi\colon\mathbb{R}^{4d}\to\mathbb{R}^{2d},\quad(q,x,p,y) \mapsto \left(\frac{q+x}{2},\frac{p+y}{2}\right)$$
to yield an explicit integrator on the original phase space. These methods were called `symplectic-like' `because they, like standard implicit symplectic integrators, show no drift in the energy error' \cite{luo}. This lack of energy drift has been observed in many astrophysical simulations with nonseparable Hamiltonians without explanation, and the method has become quite popular \cite{li,pan}.
The order can be increased using composition methods \cite{mc-qu}.
The following result accounts for the greatly reduced energy drift. 

\begin{proposition}
The $s$-stage midpoint projected methods of the form
$$\varphi_h := \pi \circ \prod_{i=1}^s \Phi_{\alpha_i h}$$
 are equivalent to $2s+1$-stage explicit Runge--Kutta methods of at least the same classical order as the underlying composition method. In the  three cases
 \begin{enumerate}
 \item $s=1$, $\alpha_1=1$ (the standard extended phase space integrator with midpoint projection, of classical order 2);
 \item $s=3$, $(\alpha_1,\alpha_2,\alpha_3)=(\alpha,1-2\alpha,\alpha)$, $\alpha=1/(2-2^{1/3})$ (classical order 4);
 \item $s=5$, $(\alpha_1,\dots,\alpha_5)=(\alpha,\alpha,1-4\alpha,\alpha,\alpha)$, $\alpha=1/(4-4^{1/3})$ (classical order 4).
 \end{enumerate}
the methods have pseudosymplecticity order $k:=5$, $9$, and $9$ respectively, and pseudosymmetry order $5$, $9$, and $9$ respectively.
That is, $\varphi_h^*\omega = \omega + \mathcal{O}(h^{k+1})$  and $\varphi_h\circ\varphi_{-h} = id + \mathcal{O}(h^{k+1})$.
\end{proposition}
\begin{proof}
We begin by noting that the extended phase space methods do not rely on the partitioning into $(q,p)$ variables, but can be written in an affine-equivariant way that exhibits how they can be applied to any ordinary differential equation. For the ODE $\dot z = f(z)$, $z\in\mathbb{R^d}$, we consider the duplicated (i.e. extended) system
\begin{align}
\label{eq:ext}
\begin{split}
\dot z &= f(\hat z),\\
\dot{\hat z} &= f(z)
\end{split}
\end{align}
which is separable and can be integrated by splitting and composition. When $z=(q,y)$, $\hat z=(x,p)$, and $f(z)=X_H(z)$, this yields the method above. When $f$ preserves $dz\wedge Jz$, \eqref{eq:ext} preserves $dz\wedge Jd\hat z$.
We write out the method $\pi\circ \Phi_h$ first in $(z,\hat z)$ variables, with initial conditions $(z_0,\hat z_0)$:
\begin{align*}
z_{1/2} &= z_0 + \frac{1}{2}h f(\hat z_0)\\
\hat z_1 &= \hat z_0 + h f(z_1) \\
z_1 &= z_{1/2} + \frac{1}{2}h f(\hat z_1)\\
\pi(z_1,\hat z_1) &= (z_1+\hat z_1)/2
\end{align*}
Imposing $\hat z_0 = z_0$, this can be written in Runge--Kutta form as
\begin{align*}
Z_1 &= z_0\\
Z_2 &= z_0 + \frac{1}{2}h f(Z_1)\\
Z_3 &= z_0 + h f(Z_2)\\
Z_4 &= Z_2 + \frac{1}{2}h f(Z_3)\\
 &= z_0 + h \left(\frac{1}{2}f(Z_1) + \frac{1}{2}f(Z_2)\right)\\
z_1 &= (Z_3+Z_4)/2\\
   &= z_0 + h \left(\frac{1}{4}f(Z_1) + \frac{1}{2}f(Z_2) + \frac{1}{4}f(Z_3)\right)
   \end{align*}
This is a 3-stage explicit Runge--Kutta method with Butcher tableau
$$
\renewcommand\arraystretch{1.2}
\begin{array}{c|ccc}
0 & 0 & 0 & 0 \\
\frac{1}{2} & \frac{1}{2} & 0 & 0 \\
1 & 0 & 1 & 0 \\
\hline
 & \frac{1}{4} & \frac{1}{2} & \frac{1}{4} \\
 \end{array}.
$$
For this method we compute its B-series and check the pseudosymplecticity conditions%
\footnote{We evaluated the symplecticity conditions $a(u)a(v)-a(u\circ v)-a(v\circ u)$ where $u$ and $v$ are Butcher trees in Mathematica using
\begin{verbatim}
<<NumericalDifferentialEquationAnalysis`
ButcherProduct[u_, v_] := If[ByteCount[v]==0, \[FormalF][u], ReplacePart[v, 1->v[[1]] u]]
Symplectic[u_, v_] := ButcherPhi[u] ButcherPhi[v] - ButcherPhi[ButcherProduct[u, v]] 
                      - ButcherPhi[ButcherProduct[v, u]]
\end{verbatim}
}
 \cite[VI.7.3]{hlw}. There are 1, 1, 1, 3, and 6 conditions respectively at orders 1, \dots, 5; these are all satisfied. Of the 16 order 6 conditions, 13 are satisfied and 3 are not, thus the method is pseudosymplectic of order 5%
. For pseudosymmetry, we expand $\varphi_h\circ \varphi_{-h}$ in B-series similarly.
 
 The calculation for the other methods proceeds similarly. For $s=3$ the Butcher tableau is
$$
\renewcommand\arraystretch{1.2}
\begin{array}{c|ccccccc}
0 & 0 & 0 & 0 & 0 & 0 & 0 & 0  \\
\frac{1}{2}\alpha_1 & \frac{1}{2}\alpha_1  & 0 & 0 & 0 & 0 & 0 & 0  \\
\alpha_1& 0 & \alpha_1 & 0 & 0 & 0 & 0 & 0  \\
\alpha_1+\frac{1}{2}\alpha_2 & \frac{1}{2}\alpha_1  & 0 & \frac{1}{2}\alpha_1 +\frac{1}{2}\alpha_2  & 0 & 0 & 0 & 0  \\
\alpha _1+\alpha_2&  0 & \alpha_1 & 0 & \alpha_2 & 0 & 0 & 0  \\
\alpha _1+\alpha _2+\frac{1}{2}\alpha_3&  \frac{1}{2}\alpha_1  & 0 & \frac{1}{2}\alpha_1 +\frac{1}{2}\alpha_2  & 0 & \frac{1}{2}\alpha_2 +\frac{1}{2}\alpha_3  & 0 & 0  \\
\alpha _1+\alpha _2+\alpha _3&  0 & \alpha_1 & 0 & \alpha_2 & 0 & \alpha_3 & 0  \\
%\alpha_1+\alpha _2+\alpha _3 & \frac{1}{2}\alpha_1  & 0 & \frac{1}{2}\alpha_1 +\frac{1}{2}\alpha_2  & 0 & \frac{1}{2}\alpha_2 +\frac{1}{2}\alpha_3  & 0 & \frac{1}{2}\alpha_3  & 0 \\
   \hline
   & \frac{1}{4}\alpha_1&\frac{1}{2}\alpha_1 &\frac{1}{4} \left(\alpha_1+\alpha_2\right)&\frac{1}{2}\alpha _2&\frac{1}{4} \left(\alpha_2+\alpha_3\right)&\frac{1}{2}\alpha _3&\frac{1}{4}\alpha_3\\

\end{array}
$$
\end{proof}

Pseudosymplecticity was introduced by Aubry and Chartier \cite{chartier}, who derive various explicit pseudosymplectic Runge--Kutta methods; the second order method above appears there as a member of a 1-parameter family of 3-stage methods of pseudosymplectic order 4, it being the unique member of that family that is pseudosymplectic of order 5.

The numerical examples in the astrophysics literature do not show energy drift. However, the potential drift effect is rather small and we have confirmed numerically in planar Hamiltonian systems that the energy does drift proportionally to $h^5 t$ resp. $h^9 t$ for the 2nd resp. 4th order methods above.  Symmetry also affects time integration and can moderate energy drift \cite{mc-pe}, so it is possible that pseudosymmetry is having some positive effect as well. The 3-stage 4th order method is known to have large error constants and it is possible that the 5-stage method given here, or some other explicit pseudosymplectic method, may be have advantages in these applications. On the other hand, energy behaviour is not the only manifestation of symplecticity and the choice of a symplectic vs a pseudosymplectic integrator may depend on the application.

\section{The symmetric projection method} 

Ohsawa \cite{ohsawa} defined the {\em extended phase space integrator with symmetric projection} as follows. 
Let
$$ D = \begin{bmatrix} I_d & -I_d & 0 & 0 \\ 0 & 0 & I_d & -I_d \end{bmatrix},\quad\mathcal{N}=\mathrm{ker} D = \{(q,q,p,p)\colon(q,p)\in\mathbb{R}^{2d}\}.$$
Given an extended phase space integrator $\Phi_h\colon \mathbb{R}^{4d}\to\mathbb{R}^{4d}$ and initial condition $z_0=(q_0,p_0)\in\mathbb{R}^{2d}$, the integrator is the map $z_0\mapsto z_1$ defined by
\begin{enumerate}
\item $ \zeta_0 := (q_0,q_0,p_0,p_0) $
\item Find $\mu\in\mathbb{R}^{2d}$ such that $\Phi_h(\zeta_0+D^\top\mu)+D^\top\mu\in\mathcal{N}$
\item $\hat\zeta_0 := \zeta_0 + D^\top\mu$
\item $\hat\zeta_1 := \Phi_h(\hat\zeta_0)$
\item $\zeta_1 = (q_1,q_1,p_1,p_1) := \hat\zeta_1 + D^\top \mu$
\item $z_1 = (q_1,p_1)$
\end{enumerate}
The method is well-defined, symmetric, and symplectic \cite{ohsawa}.
Jayawardana and Ohsawa \cite{ohsawa2} further show that the method preserves arbitrary quadratic invariants. Together these results give a strong indication that the method may be a symplectic B-series method. We  show below that this is true and that it is in fact equivalent to a monoimplicit Runge--Kutta method, a class introduced by Cash \cite{cash} that have only a single implicit stage. (The Simpson--AVF method
$z_1 = z_0 + \frac{1}{6}h(f(z_0) + 4 f((z_0+z_1)/2) + f(z_1))$ is an example \cite{ce}.) In a monoimplicit Runge--Kutta method the $A$ matrix of the Butcher tableau takes the form of a rank one matrix plus a matrix of zero spectral radius.

The extended phase space integrator with symmetric projection is also an instance of the class of {\em extended Runge--Kutta methods} that we now define.

\begin{definition}
For ODEs $\dot z = f(z)$, $z\in\mathbb{R}^d$, an extended Runge--Kutta method is defined by the equations
\begin{align*}
Z_i &= z_0 + h \sum_{j=1}^m a_{ij} k_j,\quad i=1,\dots,s \\
z_1 =& z_0 + h \sum_{i=1}^m b_i k_i\\
k_i &= f(Z_i),\quad i=1,\dots,s\\
0 &= \sum_{j=1}^{m} d_{ij}k_j,\quad i=1,\dots,m-s
\end{align*}
where the matrix $d_{ij}$ has full rank.
\end{definition}

Recall the central result of Munthe-Kaas et al. \cite{mk}: 
Let $\Phi = \{\Phi_n\}_{n\in\mathbb{N}}$ be an integration method, defined for all vector fields on all
dimensions $n$. Then $\Phi$ is a B-series method if and only if the property of affine
equivariance is fulfilled: if $a(x) := Ax + b$ is an affine map from $\mathbb{R}^m$ to $\mathbb{R}^n$, $f$ a vector field on $\mathbb{R}^m$, and $g$ a vector field on $\mathbb{R}^n$ such that
$g(Ax + b) = Af(x)$, then $a \circ\Phi_m(f) = \Phi_n(f)\circ a$. 

\begin{proposition}
Extended Runge--Kutta methods are affine equivariant.
\end{proposition}

\begin{proposition}
Let $M$ be the $m\times m$ matrix defined by
$$ M_{ij} = b_i b_j - b_i a_{ij} - b_j a_{ji}$$
for $i,j=1,\dots,m$, where we have defined $a_{ij}=0$ for $i>s$.
Let $V$ be an $m\times s$ matrix whose columns form a basis for the nullspace of $d$.
If $b_i=0$ for $i=s+1,\dots,m$, and $V^TMV=0$, then the extended Runge--Kutta method
with parameters $a_{ij}$, $b_i$, and $d_{ij}$, is quadratic-preserving and symplectic.
\end{proposition}

\begin{proof}
Suppose $f$ has a quadratic first integral $z^\top C z$. Following the standard proof for Runge--Kutta methods,

$$z_1^\top C z_1 - z_0^\top C z_0  = 2 h \sum_{j+1}^m b_i z_0^\top C k_i + h^2 \sum_{i,j=1}^m M_{ij} k_i^\top C k_j.$$
Using the condition on the $b_i$, and expressing $k_1,\dots,k_s$ in the basis whose columns are $V$, i.e. $k_i = \sum_{j=1}^s V_{ij} \hat k_j$, gives the result.
\end{proof}

\begin{proposition}
Extended  phase space integrators with symmetric projection, where the extended method $\Psi$ is a composition method, are extended Runge--Kutta methods and can be written as monoimplicit symplectic Runge--Kutta methods.
\end{proposition}

\begin{corollary}\ 
\begin{enumerate}
\item 
Extended phase space integrators with symmetric projection methods preserve arbitrary affine symmetries, quadratic integrals, and constant symplectic structures when there are any.
\item Monoimplicit symplectic Runge--Kutta methods of all orders exist.
\end{enumerate}
\end{corollary}

\begin{proof}
We again write the extended system as
\begin{align*}
\dot z &= f(\hat z)\\
\dot{\hat z} &= f(z).\\
\end{align*}
Let  $\Delta\colon z \mapsto (z,z)$; $S_\mu\colon (z,\hat z) \mapsto (z+\mu,z-\mu)$; $\pi\colon (z,\hat z)\to z$.
In these variables the symmetric projection method takes the form
$$ \pi \circ S_\mu \circ \Psi \circ S_\mu\circ \Delta$$
where $\mu$ is determined by the condition that $S_\mu \circ \Psi_h \circ S_\mu\in\mathcal{N}$.

We first illustrate the complete construction for the case that $\Psi_h$ is the leapfrog method $\Phi_h$. Its 
three substeps are then
\begin{align*}
z^* &= z_0 + \mu + \frac{1}{2} h f(z_0-\mu)\\
\hat z^* &= z_0 - \mu + h f(z^*)\\
z^{**} &= z^* + \frac{1}{2} h f(\hat z^*)
\end{align*}

Defining the three stages $Z_1$, $Z_2$, and $Z_3$ as the $z$-values at which $f$ is evaluated, and defining $h k_4 = \mu$, we have
\begin{align*}
Z_1 &= z_0 - h k_4 \\
Z_2 &= z_0 + h\left(\frac{1}{2}k_1 + k_4\right) \\
Z_3 &= z_0 + h(k_2 - k_4)
\end{align*}
The  condition $S_\mu \circ \Psi \circ S_\mu\in\Delta(\mathbb{R}^n)$ is
$$ z^{**} + \mu = \hat z^*-\mu$$
which leads to the constraint
$$ -\frac{1}{2}k_1 + k_2 - \frac{1}{2}k_3 - 4 k_4 = 0.$$
The update equation is $z_1 = z^{**}+\mu$ or
$$ z_1 = z_0 + h\left(\frac{1}{4}k_1 + \frac{1}{2}k_2 + \frac{1}{4} k_3\right).$$

Thus the parameters for this method are
$$ A = \begin{bmatrix}0 & 0 & 0 & -1 \\ \frac{1}{2} & 0 & 0 & 1 \\ 0 & 1 & 0 & -1 \\ 0 & 0 & 0 & 0 \end{bmatrix}$$
$$ b = \begin{bmatrix}\frac{1}{4} & \frac{1}{2} & \frac{1}{4} & 0 \end{bmatrix}$$
$$ d = \begin{bmatrix} -\frac{1}{2} & 1 & -\frac{1}{2} & -4 \end{bmatrix}.$$

(For the null space of $d$ we take
$$V = \begin{bmatrix} 2 & -1 & -8 \\ 1 & 0 & 0 \\ 0 & 1 & 0 \\ 0 & 0 & 1 \end{bmatrix}.$$
A direct calculation shows that 
$$M = \frac{1}{16}\begin{bmatrix}1 & -2 & 1 & 4 \\ -2 & 4 & -2 & -8 \\ 1 & -2 & 1 & 4 \\ 4 & -8 & 4 & 0\end{bmatrix}$$
and
$V^\top M V = 0$, confirming that the method is quadratic-preserving.)

The extra stage $k_4$ can be explicitly eliminated using the constraint, giving the 3-stage monoimplicit symplectic Runge--Kutta method with tableau
$$
\renewcommand\arraystretch{1.2}
\begin{array}{c|ccc}
0 & \frac{1}{8} & -\frac{1}{4} & \frac{1}{8} \\
\frac{1}{2} & \frac{3}{8} & \frac{1}{4} & -\frac{1}{8}\\
1 & \frac{1}{8} &\frac{3}{4} & \frac{1}{8}\\
\hline
& \frac{1}{4} & \frac{1}{2} & \frac{1}{4}
\end{array}.
$$

For the general case, let $\Psi_h$ be the composition method with time steps given by parameters $a_1,\dots,a_s$, i.e.
$$\Psi_h = \exp(a_s h X_{\hat H_B})\exp(a_{s-1} h X_{\hat H_A}) \dots \exp(a_2 h X_{\hat H_B})\exp(a_{1} h X_{\hat H_A}).$$
Writing out the stages as above, and eliminating the constraint $\mu$, leads to an $s$-stage monoimplicit Runge--Kutta method with parameters
$$ A = L + \frac{1}{4}u v^\top,\quad b_i = \frac{1}{2}a_i$$
where
$$ L = \begin{bmatrix} 0 & \dots\\
a_1 & 0 & \dots \\
0 & a_2 & 0 & \dots \\
a_1 & 0 & a_3 & 0 & \dots \\
0 & a_2 & 0 & a_4 & 0 & \dots \\
& & \ddots & & \ddots \\
a_1 & 0 & a_3 & \dots & & a_{s-1} & 0
\end{bmatrix}
,\quad
u_i = (-1)^i,\quad v_j= a_j (-1)^{j}.$$
Therefore
$$ b_i a_{ij} = \frac{1}{2}b_i b_j c_{j-i}$$
where
$$c_k = 
\begin{cases}
1 & k \hbox{\rm\ even} \\
-1 & k\hbox{\rm\ odd}, k>0\\
3 & k  \hbox{\rm\ odd}, k<0
\end{cases}
$$
and thus we have
$$b_i b_j - b_i a_{ij} - b_j a_{ji} = 0$$
for all $i$ and $j$, which is the condition for a Runge--Kutta method to be symplectic.
\end{proof}

\section{Discussion}
The results here have been established by direct calculation, but the methods appear quite natural and intrinsically defined. One could seek direct geometric proofs, not relying on B-series, of the pseudosymplecticity and pseudosymmetry orders for the midpoint projection method when the extended method is an arbitrary symplectic integrator. Likewise, for the symmetric projection method, the diagonal $\mathcal{N}$ is a symplectic subspace of the extended symplectic space,  
suggesting an approach using the symplectic geometry of constraints \cite{shake}. The question mentioned earlier, of the best pseudosymplectic method to use in applications, and any potential issues arising from nonsymplecticity, should be examined. Finally, we suggest determining the  entire set of monoimplicit symplectic Runge--Kutta methods and their relative merits.

\begin{paragraph}{Dedication}\ \\
A preliminary version of this work was presented at ANODE 2023 in honour of John Butcher's 90th birthday. It was therefore very pleasing and appropriate to find that, as the work developed, it turned out to involve Runge--Kutta methods and Butcher series so intimately. Happy birthday, John!
\end{paragraph}

\end{document}